\documentclass[10pt]{amsart}
\usepackage{epsfig,url}
\usepackage{mathdots}

\newtheorem{theorem}{Theorem}[section]
\newtheorem{lemma}[theorem]{Lemma}
\newtheorem{conjecture}[theorem]{Conjecture}

\newtheorem{corollary}[theorem]{Corollary}
\newtheorem{proposition}[theorem]{Proposition}

\theoremstyle{definition}

\newtheorem{note}[theorem]{Note}

\theoremstyle{remark}

\begin{document}

\title[Unimodal polynomials and quartic integrals]
{The unimodality of a polynomial coming from a rational integral. Back to the 
original proof}

\author[T. Amdeberhan et al]{Tewodros Amdeberhan, Atul Dixit, Xiao Guan, Lin Jiu, Victor H. Moll}
\address{Department of Mathematics,
Tulane University, New Orleans, LA 70118}
\email{tamdeber@tulane.edu,adixit@tulane.edu,xguan1@tulane.edu,ljiu@tulane.edu
\\vhm@tulane.edu}

%
%
%
%

\date{\today}


\begin{abstract}
A sequence of coefficients that appeared in the evaluation of a rational 
integral has been shown to be unimodal. An alternative proof 
is presented. 
\end{abstract}

\maketitle

\newcommand{\ba}{\begin{eqnarray}}
\newcommand{\ea}{\end{eqnarray}}
\newcommand{\ift}{\int_{0}^{\infty}}
\newcommand{\nn}{\nonumber}
\newcommand{\no}{\noindent}
\newcommand{\lf}{\left\lfloor}
\newcommand{\rf}{\right\rfloor}
\newcommand{\realpart}{\mathop{\rm Re}\nolimits}
\newcommand{\imagpart}{\mathop{\rm Im}\nolimits}

\newcommand{\op}[1]{\ensuremath{\operatorname{#1}}}
\newcommand{\pFq}[5]{\ensuremath{{}_{#1}F_{#2} \left( \genfrac{}{}{0pt}{}{#3}
{#4} \bigg| {#5} \right)}}

\newtheorem{Definition}{\bf Definition}[section]
\newtheorem{Thm}[Definition]{\bf Theorem}
\newtheorem{Example}[Definition]{\bf Example}
\newtheorem{Lem}[Definition]{\bf Lemma}
\newtheorem{Cor}[Definition]{\bf Corollary}
\newtheorem{Prop}[Definition]{\bf Proposition}
\numberwithin{equation}{section}

\section{Introduction}
\label{sec-intro}

The polynomial 
\begin{equation}
P_{m}(a) = \sum_{\ell=0}^{m} d_{\ell}(m)a^{\ell}
\label{poly-def}
\end{equation}
\noindent
with 
\begin{equation}
d_{\ell}(m) = 2^{-2m} \sum_{k= \ell}^{m} 2^{k} \binom{2m-2k}{m-k} 
\binom{m+k}{m} \binom{k}{\ell}
\end{equation}
\noindent
made its appearance in \cite{boros-1999c} in 
the evaluation of the quartic integral 
\begin{equation}
\int_{0}^{\infty} \frac{dx}{(x^{4}+2ax^{2}+1)^{m+1}} = 
\frac{\pi}{2^{m+3/2} (a+1)^{m+1/2}} P_{m}(a).
\end{equation}
\noindent
Properties of the sequence of numbers  $\{ d_{\ell}(m) \}$ 
are discussed in \cite{manna-2009a}. Among them 
is the fact that this is a unimodal sequence. Recall that a sequence of real 
numbers $\{ x_{0}, \, x_{1}, \, \cdots, x_{m} \}$ is called \textit{unimodal}
if there exists an index $0 \leq j \leq m$ such that $x_{0} \leq x_{1} 
\leq \cdots \leq x_{j}$ and $x_{j} \geq x_{j+1} \geq \cdots x_{m}$. The 
sequence is called \textit{logconcave} if $x_{j}^{2} \geq x_{j-1}x_{j+1}$
for $1 \leq j \leq m-1$. It is easy to see that if a sequence is logconcave 
then it is unimodal \cite{wilf-1990a}. 

The sequence $\{ d_{\ell}(m) \}$ was shown to be unimodal in \cite{boros-1999b}
by an elementary argument and it was conjectured there to be logconcave. This 
conjecture was established by M. Kauers and P. Paule \cite{kauers-2007a} 
using four recurrence relations found using a computer algebra approach. 
W. Y. Chen and E. X. W. Xia \cite{chen-2009c} introduced the notion of 
\textit{ratio-monotonicity} for a sequence $\{ x_{m} \}$:
\begin{equation}
\frac{x_{0}}{x_{m-1}} \leq \frac{x_{1}}{x_{m-2}} \leq \cdots \leq 
\frac{x_{i}}{x_{m-1-i}} \leq \cdots \leq 
\frac{x_{\lf \tfrac{m}{2} \rf -1}}{x_{m- \lf \tfrac{m}{2} \rf}} \leq 1.
\end{equation}
\noindent
The results in \cite{chen-2009c} show that $\{ d_{\ell}(m) \}$ is a 
ratio-monotone sequence and, as can be 
easily checked, this implies the logconcavity
of $\{ d_{\ell}(m) \}$. The logconcavity of $\{ d_{\ell}(m) \}$ also follows
from the \textit{minimum conjecture} stated in \cite{moll-2007e}: let 
$b_{\ell}(m) = 2^{2m} d_{\ell}(m)$. The function 
\begin{equation*}
(m+ \ell)(m+1-\ell)b_{\ell-1}^{2}(m) + \ell (\ell + 1) b_{\ell}^{2}(m) - 
\ell (2m+1)b_{\ell-1}(m),
\end{equation*}
\noindent
defined for $1 \leq \ell \leq m$, attains its minimum at $\ell = m$ with 
value $2^{2m} m(m+1) \binom{2m}{m}^{2}$. This has been proven in 
\cite{chen-2013a}, providing an alternative proof of the logconcavity of 
$\{ d_{\ell}(m) \}$. 

\smallskip

Further study of the sequence $\{d_{\ell}(m) \}$ are defined in terms of the 
operator 
\begin{equation}
\mathfrak{L} \left( \{ x_{k} \} \right) = \left\{  x_{k}^{2} - x_{k-1}x_{k+1}
\right\}.
\end{equation}
\noindent
For instance, $\{ x_{k} \}$ is logconcave simply means 
$\mathfrak{L} \left( \{ x_{k} \} \right)$ is a nonnegative sequence. The 
sequence is called $i$-logconcave if 
$\mathfrak{L}^{j} \left( \{ x_{k} \} \right)$ is a nonnegative sequence for 
$0 \leq j \leq i$. A sequence that is $i$-logconcave for every
 $i \in \mathbb{N}$ is called \textit{infinitely logconcave}. 

\begin{conjecture}
The sequence $\{ d_{\ell}(m) \}$ is infinitely logconcave.
\end{conjecture}

There is a strong connection between the roots of a polynomial $P(x)$ and 
ordering properties of its coefficients. For instance, if $P(x)$ has 
only real negative zeros, then $P$ is logconcave (see \cite{wilf-1990a} for 
details). Therefore, the expansion of $(x+1)^{n}$ shows that the binomial
coefficients form a logconcave sequence. P. Br\"{a}nd\'{e}n 
\cite{branden-2011a} showed that if 
$P(x) = a_{0} + a_{1}x + \cdots + a_{n}x^{n},$ with $a_{j} \geq 0$ has only 
real roots, then the same is true for 
\begin{equation}
P_{1}(x) = a_{0}^{2} + (a_{1}^{2}-a_{0}a_{2})x + \cdots + 
(a_{n-1}^{2}-a_{n-2}a_{n})x^{n}.
\end{equation}
\noindent 
This implies that the binomial coefficients are infinitely logconcave. This
approach fails with the sequence $\{ d_{\ell}(m) \}$ since the polynomial 
$P_{m}(a)$ has mostly non-real zeros. On the other hand, Br\"{a}nd\'{e}n
conjectured and W. Y. C. Chen et al \cite{chen-2013b} proved that 
$\begin{displaystyle}
Q_{m}(x) = \sum_{\ell=0}^{m} \frac{d_{\ell}(m)}{\ell!} x^{\ell}
\end{displaystyle}$
and 
$\begin{displaystyle}
R_{m}(x) = \sum_{\ell=0}^{m} \frac{d_{\ell}(m)}{(\ell+2)!} x^{\ell}
\end{displaystyle}$
have only real zeros. These results imply that $P_{m}(a)$ in \eqref{poly-def}
is $3$-logconcave. 

The goal of this paper is to present an improved version of
the original proof of the theorem 

\begin{theorem}
\label{thm-uni}
The sequence $\{ d_{\ell}(m) \}$ is unimodal.
\end{theorem}

The proof of Theorem \ref{thm-uni} given in \cite{boros-1999b} is 
based on the difference 
\begin{equation}
\Delta d_{\ell}(m) = d_{\ell+1}(m) - d_{\ell}(m).
\end{equation}
\noindent
A simple calculation shows that 
\begin{equation}
\Delta d_{\ell}(m) = \frac{1}{2^{2m}} \binom{m+ \ell}{m} 
\sum_{k=\ell}^{m} 2^{k} \binom{2m-2k}{m-k} \binom{m+k}{m+ \ell} 
\times \frac{k - 2 \ell -1}{\ell+1}.
\end{equation}
\noindent
For $\lf \frac{m}{2} \rf \leq \ell \leq m- 1$, the inequality 
\begin{equation}
k - 2 \ell -1 \leq k - 2 \lf \frac{m}{2} \rf - 1 \leq k - m \leq 0
\end{equation}
\noindent
shows that $\Delta d_{\ell}(m) < 0$ since the term for $k = \ell$ has 
a strictly negative contribution.  In the 
range $0 \leq \ell < \lf \frac{m}{2} \rf$, the difference 
$\Delta d_{\ell}(m) > 0$. This is equivalent to
\begin{equation}
\label{task-1}
\sum_{k= \ell}^{2 \ell} 2^{k} (2 \ell + 1 - k) \binom{2m-2k}{m-k} 
\binom{m+k}{m+ \ell} < 
\sum_{k= 2 \ell + 2}^{m} 2^{k} (k-2 \ell - 1) \binom{2m-2k}{m-k} 
\binom{m+k}{m+ \ell}.
\end{equation}

\smallskip

\noindent
\textbf{Fact 1}. The inequality \eqref{task-1} implies Theorem \ref{thm-uni}.

\smallskip

\noindent
This required inequality is valid in an even stronger form, obtained by 
replacing $k-2 \ell - 1$ on the right hand side of \eqref{task-1} by $1$ to 
produce 
\begin{equation}
\label{task-2}
\sum_{k= \ell}^{2 \ell} 2^{k} (2 \ell + 1 - k) \binom{2m-2k}{m-k} 
\binom{m+k}{m+ \ell} < 
\sum_{k= 2 \ell + 2}^{m} 2^{k} \binom{2m-2k}{m-k} 
\binom{m+k}{m+ \ell},
\end{equation}
\noindent 
and then made even stronger by replacing the sum on the right hand side 
of \eqref{task-2} by its last term. Therefore, if 
\begin{equation}
\label{task-3}
\sum_{k= \ell}^{2 \ell} 2^{k} (2 \ell + 1 - k) \binom{2m-2k}{m-k} 
\binom{m+k}{m+ \ell} < 
2^{m} \binom{2m}{m+ \ell},
\end{equation}
then $\Delta d_{\ell}(m) > 0$.  This last inequality is now 
written as 
\begin{equation}
S_{m, \ell} := 
\sum_{k= \ell}^{2 \ell} \binom{m - \ell}{m - k } 
\binom{m+k}{2k} \binom{2m}{2k}^{-1} \times 
\frac{2 \ell + 1 - k}{2^{m-k}} < 1.
\label{task-4}
\end{equation}

\smallskip

\noindent
\textbf{Fact 2}. The inequality \eqref{task-4} implies Theorem \ref{thm-uni}.

\smallskip

In \cite{boros-1999b}, the proof of \eqref{task-4} is divided into 
two parts: first 

\begin{theorem}
\label{missing-1}
For fixed $m \in \mathbb{N}$ and $0 \leq \ell < \lf \frac{m}{2} \rf$, the 
sum $S_{m, \ell}$ is increasing in $\ell$. 
\end{theorem}

\noindent
and then 

\begin{theorem}
\label{missing-2}
The maximal sum $S_{m, \lf \tfrac{m-1}{2} \rf}$ is strictly less than $1$. 
For $m$ even, the maximal sum $S_{2m,m-1}$ is given by 
\begin{equation}
T(m) := S_{2m,m-1} = \sum_{r=2}^{m+1} \binom{2r}{r} \binom{m+1}{r} 
\frac{(r-1)}{2^{r} \binom{4m}{r}}.
\end{equation}
\noindent
a similar expression exists for $m$ odd.
\end{theorem}

\noindent
\textbf{Fact 3}. Theorems \ref{missing-1} and \ref{missing-2} imply
Theorem \ref{thm-uni}.

\smallskip

These two results were established in \cite{boros-1999b} by some 
elementary estimates. These were long and do not extend to, for instance, the 
proof of logconcavity of $\{ d_{\ell}(m) \}$. The hope is that the techniques 
used to provide the new proof of unimodality presented here, will also 
apply to other situations. 

Section \ref{bound-T} presents a new 
elementary proof of Theorem 
\ref{missing-2} and Section \ref{miss-2} contains a proof based on a 
hypergeometric representation of $T(m)$. Section \ref{sec-limiting} shows 
that $T(m)$ converges to the value 
\begin{equation}
\lim\limits_{m \to \infty} 
\sum_{r=2}^{m+1} \binom{2r}{r} \binom{m+1}{r} 
\frac{(r-1)}{2^{r} \binom{4m}{r}} =  \frac{2-\sqrt{2}}{2} \sim 0.292893.
\label{limit-one}
\end{equation}
\noindent
This limit was \textit{incorrectly} conjectured in \cite{boros-1999b} 
to be $1- \ln 2 \sim 0.306853$.  The authors have failed to produce a proof 
of Theorem \ref{missing-1} by the automatic techniques developed in 
\cite{petkovsek-1996a}. These methods yield recurrences for the summands 
in \eqref{task-4}, but it is not possible to conclude from them 
that $S_{m,\ell}$ is increasing.  The last section shows that the sequence 
$\{ T(m):  \, m \geq 2 \}$ is an increasing sequence.

\section{The bound on $T(m)$}
\label{bound-T}

The result stated in Theorem \ref{missing-2} is 
equivalent to the bound 
\begin{equation}
\label{value-tn}
T(m) := 
\sum_{r=2}^{m+1} \binom{2r}{r} \binom{m+1}{r} \frac{(r-1)}{2^{r} \binom{4m}{r}}
< 1.
\end{equation}

A direct proof of this result is given  next. Section \ref{miss-2} presents 
a proof based on a hypergeometric representation of $T(m)$. 

\begin{theorem}
\label{bound-tm}
The inequality $T(m) < 1$ holds.
\end{theorem}
\begin{proof}
First, it is shown by induction that for $m$ fixed and $2 \leq r \leq m+1$
\begin{equation}
a_m(r):=\binom{2r}r\binom{m+1}r \leq b_m(r):=\binom{4m}{r}.
\end{equation}
\noindent
If $r=2$: $b_m(2)-a_m(2)=5m(m-1) \geq 0$. Now observe that 
\begin{equation*}
\frac{b_m(r+1)}{b_m(r)}-\frac{a_m(r+1)}{a_m(r)}=
\frac{4m-r}{r+1}-\frac{2(2r+1)(m+1-r)}{(r+1)^2}=
\frac{2(m-1)+3r(r-1)}{(r+1)^2}>0.
\end{equation*}
\noindent
This gives the inductive step written as
\begin{equation*}
b_m(r)\frac{b_m(r+1)}{b_m(r)}>a_m(r)\frac{a_m(r+1)}{a_m(r)}.
\end{equation*}
\noindent
The inequality $a_m(r)<b_m(r)$ now yields
\begin{equation*}
T(m)=
\sum_{r=2}^{m+1}\frac{a_m(r)}{b_m(r)}\frac{r-1}{2^r}
< \sum_{r=2}^{m+1}\frac{r-1}{2^r}=1-\frac{m+2}{2^{m+1}}<1.
\end{equation*}
\end{proof}

\section{The hypergeometric representation of $T(m)$}
\label{miss-2}

This section provides a hypergeometric representation of the sum 
\begin{equation}
T(m) = \sum_{r=2}^{m+1} \binom{2r}{r} \binom{m+1}{r} 
\frac{(r-1)}{2^{r} \binom{4m}{r}}.
\end{equation}
\noindent

\begin{proposition}
\label{prop-hypertn}
The sum $T(m)$ is given by 
\begin{equation}
T(m) = 1 - \pFq21{\tfrac{1}{2}, -1-m}{-4m}{2} + 
\frac{m+1}{4m} \pFq21{\tfrac{3}{2}, -m}{1-4m}{2}.
\label{hyp-rep}
\end{equation}
\end{proposition}
\begin{proof}
Since $\begin{displaystyle} \binom{m}{k} = \frac{(-1)^{k} (-m)_{k}}{k!} 
\end{displaystyle}$, it follows that 
$\begin{displaystyle}
\frac{\binom{m+1}{r}}{\binom{4m}{r}} = \frac{(-1-m)_{r}}{(-4m)_{r}}
\end{displaystyle}$. This relation and 
$\begin{displaystyle} \left( \tfrac{1}{2} \right)_{r} = 
(2r)!/(2^{2r} \, r!) \end{displaystyle}$
\noindent 
give
\begin{equation}
T(m) = \sum_{r=2}^{m+1} \frac{\left(\tfrac{1}{2} \right)_{r}}{r!} 
\frac{(r-1)2^{r}  
(-1-m)_{r}}{(-4m)_{r}}.
\end{equation}
\noindent
Therefore 
\begin{eqnarray*}
T(m) & = & 
- \sum_{r=2}^{m+1} \frac{ \left( \tfrac{1}{2} \right)_{r} 
(-1-m)_{r} 2^{r}}{(-4m)_{r} \, r!} + 
\sum_{r=2}^{m+1} \frac{ \left( \tfrac{1}{2} \right)_{r} 
(-1-m)_{r} 2^{r}}{(r-1)! \, (-4m)_{r}}   \\
& = & 
1 + \frac{m+1}{4r} -
\sum_{r=0}^{m+1} \frac{ \left( \tfrac{1}{2} \right)_{r} (-1-m)_{r} 2^{r}}
{(-4m)_{r} \, r!} + 
\frac{m+1}{4m} 
\sum_{r=0}^{m+1} \frac{ \left( \tfrac{1}{2} \right)_{r} (-1-m)_{r} 2^{r}}
{(-4m)_{r} \, (r-1)!} \frac{4m}{m+1} \\
& = & 
1 - 
\sum_{r=0}^{m+1} \frac{ \left( \tfrac{1}{2} \right)_{r} (-1-m)_{r}}
{(-4m)_{r} } \frac{2^{r}}{r!} + 
\frac{m+1}{4m} 
\left\{ 1 + \sum_{r=2}^{m+1} \frac{ \left( \tfrac{1}{2} \right)_{r} 2^{r}}
{(r-1)!} \frac{4m}{m+1} \, \frac{(-1-4m)_{r}}{(-4m)_{r}} \right\} \\
& = & 1 - \pFq21{\tfrac{1}{2},-1-m}{-4m}{2} + 
\frac{m+1}{4m} 
\sum_{r=2}^{m+1} \frac{ \left( \tfrac{1}{2} \right)_{r} 2^{r}}{(r-1)!} 
\frac{4m}{m+1} \frac{(-1-m)_{r}}{(-4m)_{r}} \\
& = & 1 - \pFq21{\tfrac{1}{2},-1-m}{-4m}{2} + 
\frac{m+1}{4m} \sum_{r=0}^{m} \frac{ \left( \tfrac{3}{2} \right)_{r}}{r!} 
\frac{2^{r} \, (-m)_{r}}{(1-4m)_{r}} \\
& = & 1 - \pFq21{\tfrac{1}{2},-1-m}{-4m}{2} + 
\frac{m+1}{4m}  \pFq21{\tfrac{3}{2},-m}{1-4m}{2}.
\end{eqnarray*}

\end{proof}

The next result provides an integral representation 
for $T(m)$. 

\begin{proposition}
\label{integ-tn}
The sum $T(m)$ is given by 
\begin{equation}
T(m) = \frac{3(m+1)}{16(4m-1)} \int_{0}^{2} 
t \, \pFq21{\tfrac{5}{2}, 1-m}{2-4m}{t} \, dt.
\end{equation}
\end{proposition}
\begin{proof}
Integrate by parts and use
\begin{equation}
\frac{d}{dt} \pFq21{a,b}{c}{t} = \frac{ab}{c} \pFq21{a+1,b+1}{c+1}{t}
\label{der-hyp}
\end{equation}
\noindent 
to produce 
\begin{equation*}
\int_{0}^{2} t \, \pFq21{\tfrac{5}{2},1 - m }{2 - 4m}{t} \, dt = 
\frac{4(4m-1)}{3m} \pFq21{\tfrac{3}{2},-m}{1-4m}{2} - 
\frac{2(4m-1)}{3m} 
\int_{0}^{2} \pFq21{\tfrac{3}{2}, - m }{1 - 4m}{t} \, dt.
\end{equation*}
\noindent
The last integral is evaluated using \eqref{der-hyp} to write 
\begin{equation*}
\pFq21{\tfrac{3}{2}, - m }{1 - 4m}{t}  = \frac{8m}{m+1} 
\frac{d}{dt} 
\pFq21{\tfrac{1}{2}, -1- m }{ - 4m}{t}
\end{equation*}
\noindent
and the result follows. 
\end{proof}

The next result provides a bound for the integrand in 
Proposition \ref{integ-tn}.

\begin{proposition}
\label{bound-integ-tn}
Let $n \in \mathbb{N}, \, n \geq 2$ and $0 \leq t \leq 2$. Then
\begin{equation*}
\left| \pFq21{\tfrac{5}{2}, 1-m}{2-4m}{t}  \right| 
\leq 9 \sqrt{3}(3-t)^{-5/2}.
\end{equation*}
\end{proposition}
\begin{proof}
The hypergeometric function is given by
\begin{equation*}
\pFq21{\tfrac{5}{2},1-m}{2-4m}{t}   = 
\sum_{k=0}^{m-1} \left( \frac{5}{2} \right)_{k} 
\frac{(1-m)_{k}}{(2-4m)_{k}}.
\end{equation*}
The bound 
\begin{equation}
\frac{(1-m)_{k}}{(2 - 4m)_{k}} \leq \frac{1}{3^{k}}
\label{bound-111}
\end{equation}
\noindent
follows directly from the observation that $b_{k}(m) = 
3^{k}(1-m)_{k}/(2-4m)_{k}$ satisfies $b_{0}(m) = 1$ and it 
is decreasing in $k$. Indeed,
\begin{equation}
\frac{b_{k+1}(m)}{b_{k}(m)} = \frac{3(1-m+k)}{2-4m+k} < 1.
\end{equation}
\noindent 
Then \eqref{bound-111} gives 
\begin{eqnarray*}
\pFq21{\tfrac{5}{2},1-m}{2-4m}{t}  & \leq  & 
\sum_{k=0}^{m-1} \left( \frac{5}{2} \right)_{k} \frac{t^{k}}{3^{k} k!}  \\
& \leq  &  \sum_{k=0}^{\infty} \left( \frac{5}{2} \right)_{k} \frac{(t/3)^{k}}
{k!} \\
& = &  \pFq10{\tfrac{5}{2}}{-}{\frac{t}{3}}. 
\end{eqnarray*}
\noindent
The evaluation of the final hypergeometric sum comes from the binomial theorem 
\begin{equation}
\pFq10{a}{-}{z} = (1-z)^{-a}, \text{ for } |z| < 1. 
\label{bin-theorem}
\end{equation}
\end{proof}

The bound in Theorem \ref{missing-2} is now obtained. 

\begin{corollary}
For $m \in \mathbb{N}$, the function $T(m)$ satisfies 
\begin{equation}
T(m) < 1.
\end{equation}
\end{corollary}
\begin{proof}
It is easy to compute that $T(1) = \tfrac{1}{4}$. For $m \geq 2$, observe that 
\begin{equation}
\frac{3(m+1)}{16(4m-1)} = \frac{3}{16} \left( \frac{1}{4} + 
\frac{5/4}{4m-1} \right) 
\leq \frac{9}{112}
\end{equation}
\noindent 
and thus
\begin{equation}
T(m) \leq \frac{9}{112} \int_{0}^{2} \frac{9 \sqrt{3} \, t \, dt}
{(3 - t)^{5/2}} = \frac{27}{28} < 1.
\end{equation}
\end{proof}

\begin{note}
This inequality completes the proof that $\{ d_{\ell}(m) \}$ is 
unimodal. 
\end{note}

\section{The limiting behavior of $T(m)$}
\label{sec-limiting}

This section is devoted to establish the limiting value of $T(m)$.

\begin{theorem}
\label{limit-tn}
The function $T(m)$ satisfies 
\begin{equation}
\lim\limits_{m \to \infty} T(m) = \frac{2 -\sqrt{2}}{2}. 
\end{equation}
\end{theorem}

The arguments will employ the classical Tannery theorem. This 
is stated next, a proof appears in \cite{bromwich-1991a}, page 136.

\begin{theorem}
(\textit{Tannery}) Assume
$\begin{displaystyle} 
a_{k} := \lim\limits_{m \to \infty} a_{k}(m)
\end{displaystyle}$ satisfies $|a_{k}(m)| 
\leq M_{k}$ with $\begin{displaystyle} \sum_{k=0}^{\infty} M_{k} < \infty
\end{displaystyle}$. 
Then 
$\begin{displaystyle} \lim \limits_{m \to \infty} \sum_{k=0}^{m} a_{k}(m) = 
\sum_{k=0}^{\infty} a_{k}.\end{displaystyle}$
\end{theorem}

\medskip

Three proofs of Theorem \ref{limit-tn} are presented here.  In each 
one of them, the argument boils down 
to an exchange of limits. The first one is based on the integral 
representation of $T(m)$ and it uses bounded convergence theorem and 
Tannery's theorem. The second 
one deals directly with the hypergeometric sums and it employs 
Tannery's theorem for passing to the limit in a series. A similar 
argument can be employed in the third proof.

\begin{proposition}
\label{prop-4.3}
Assume $0 \leq t < 4$ is fixed. Then
\begin{equation}
\lim\limits_{m \to \infty} 
\pFq21{\tfrac{5}{2}, 1-m}{2-4m}{t} = 
\pFq10{\tfrac{5}{2}}{-}{\frac{t}{4}} = \frac{32}{(4-t)^{5/2}}.
\end{equation}
\end{proposition}
\begin{proof}
Start with 
\begin{equation}
\pFq21{\tfrac{5}{2}, 1-m}{2-4m}{t} = \sum_{k=0}^{m-1} 
\frac{\left( \tfrac{5}{2} \right)_{k} \, (1-m)_{k}}{(2-4m)_{k} } 
\, \frac{t^{k}}{k!}
\end{equation}
\noindent
and observe that 
\begin{equation}
\frac{(1-m)_{k}}{(2-4m)_{k}} = \prod_{j=0}^{k-1} \frac{m-1-j}{4m-2-j} 
\to \frac{1}{4^{k}}
\end{equation}
\noindent 
as $m \to \infty$.  Therefore 
\begin{equation}
\lim\limits_{m \to \infty} 
\pFq21{\tfrac{5}{2}, 1-m}{2-4m}{t}  =  
\sum_{k=0}^{\infty} \frac{ \left( \tfrac{5}{2} \right)_{k} }{k!} 
\left( \frac{t}{4} \right)^{k} 
= \pFq10{\tfrac{5}{2}}{-}{\frac{t}{4}}.
\label{limit-1}
\end{equation}
\noindent
The hypergeometric sum is now evaluated using \eqref{bin-theorem}.
\end{proof}

The passage to the limit in \eqref{limit-1} uses the Tannery's theorem. In
this case 
\begin{equation}
a_{k}(m) = \frac{ \left( \tfrac{5}{2} \right)_{k} (1-m)_{k}}{(2-4m)_{k}} 
\frac{t^{k}}{k!}
\end{equation}
\noindent
satisfies 
\begin{eqnarray*}
\lim\limits_{m \to \infty} a_{k}(m) & = & 
\lim\limits_{m \to \infty}  
\left( \frac{5}{2} \right)_{k} \frac{t^{k}}{k!} 
\frac{ \left( \tfrac{1}{m}-1 \right) 
\left( \tfrac{2}{m}-1 \right)  \cdots 
\left( \tfrac{k}{m}-1 \right) }
{ \left( \tfrac{2}{m}-4 \right) 
\left( \tfrac{3}{m}-1 \right)  \cdots 
\left( \tfrac{1+k}{m}-4 \right) } \\
& = & 
\left( \frac{5}{2} \right)_{k} \frac{t^{k}}{k! \, 4^{k}}
\end{eqnarray*}
\noindent
exists. This limit is denoted by $a_{k}$.  

The result now follows from the bound 
\begin{equation}
|a_{k}(m)| \leq M_{k} := \left( \frac{5}{2} \right)_{k} \frac{t^{k}}{k! \, 
3^{k}},
\label{bound-tan1}
\end{equation}
\noindent
and the sum 
\begin{equation}
\sum_{k=0}^{\infty} M_{k} = \sum_{k=0}^{\infty} \left( \frac{5}{2} \right)_{k}
\frac{t^{k}}{k! \, 3^{k}} = \left( 1 - \frac{t}{3} \right)^{-5/2}.
\end{equation}
valid for $0 \leq t \leq 2$. Tannery's theorem gives 
\begin{equation}
\lim\limits_{m \to \infty} \sum_{k=0}^{m-1} a_{k}(m) = 
\sum_{k=0}^{\infty} a_{k}  = 
\sum_{k=0}^{\infty} \left( \frac{5}{2} \right)_{k} \frac{t^{k}}{k! \, 4^{k}} 
 = \left( 1 - \frac{t}{4} \right)^{-5/2}.
\end{equation}

%

The expression in Proposition \ref{integ-tn}, the bound 
\eqref{bound-111} and  Proposition 
\ref{bound-integ-tn} give, via the dominated convergence theorem, the 
value
\begin{eqnarray}
\lim\limits_{m \to \infty} T(m)  & = & 
\lim\limits_{m \to \infty}
\frac{3(m+1)}{16(4m-1)} \int_{0}^{2} 
\pFq21{\tfrac{5}{2}, 1-m}{2-4m}{t} \, t \, dt \label{limit-2} \\
& = & 
\frac{3}{64} \int_{0}^{2} 
\pFq10{\tfrac{5}{2}}{-}{\frac{t}{4}} \, dt
\nonumber \\
& = & 
\frac{3}{64} \int_{0}^{2} \frac{32 t}{(4-t)^{5/2}} \, dt \nonumber \\
& = & \frac{2 - \sqrt{2}}{2}. \nonumber
\end{eqnarray}

\noindent
This completes the first proof. 

\smallskip

The second proof of the limiting value of $T(m)$ uses the hypergeometric 
representation of $T(m)$ in \eqref{hyp-rep}. It amounts to proving
\begin{equation}
\lim\limits_{m \to \infty} \pFq21{\tfrac{1}{2}, -1-m}{-4m}{2}  - 
\frac{m+1}{4m} \pFq21{\tfrac{3}{2}, -m}{1-4m}{2} = \frac{\sqrt{2}}{2}.
\end{equation}

The contiguous relation \cite{temme-1996a}, page 28,
\begin{equation}
\pFq21{a+1,b}{c}{z} = 
\pFq21{a,b}{c}{z}  + 
\frac{bz}{c} \pFq21{a+1,b+1}{c+1}{z}
\label{contig-1}
\end{equation}
is used with $a = \tfrac{1}{2}, \, b = -1-m, 
\, c = -4m$ and $z=2$ to obtain 
\begin{equation*}
\pFq21{\tfrac{3}{2},-1-m}{-4m}{2} = 
\pFq21{\tfrac{1}{2},-1-m}{-4m}{2}  + 
\frac{m+1}{2m} \pFq21{\tfrac{3}{2},-m}{1-4m}{2}
\end{equation*}
\noindent
and this gives 
\begin{equation}
\frac{(m+1)}{4m} \pFq21{\tfrac{3}{2},-m}{1-4m}{2}
= \frac{1}{2} \left( 
\pFq21{\tfrac{3}{2},-1-m}{-4m}{2} - 
\pFq21{\tfrac{1}{2},-1-m}{-4m}{2}  
\right).
\label{formula-last}
\end{equation}
\noindent
Thus if suffices to prove 
\begin{equation}
\lim\limits_{m \to \infty} 
3 \,\, \pFq21{\tfrac{1}{2},-1-m}{-4m}{2} - 
\pFq21{\tfrac{3}{2},-1-m}{-4m}{2}   = \sqrt{2}.
\end{equation}
\noindent
A direct calculation shows that 
\begin{equation*}
3 \,\, \pFq21{\tfrac{1}{2},-1-m}{-4m}{2} - 
\pFq21{\tfrac{3}{2},-1-m}{-4m}{2}   = 
\sum_{k=0}^{m+1} a_{k}(m)
\end{equation*}
\noindent
with 
\begin{equation}
\label{def-ak}
a_{k}(m) = 
\sum_{k=0}^{m+1} \frac{ \left[ 3 \left( \tfrac{1}{2} \right)_{k} - 
\left( \tfrac{3}{2} \right)_{k} \right] (-1-m)_{k} 2^{k}}
{(-4m)_{k} \, k!}.
\end{equation}
\noindent
The question is now reduced to justifying passing to the limit in 
\begin{equation}
\lim\limits_{m \to \infty} \sum_{k=0}^{m+1} a_{k}(m) 
= \sum_{k=0}^{\infty} \lim\limits_{m \to \infty} a_{k}(m) 
\end{equation}
\noindent
since 
\begin{equation}
\lim\limits_{m \to \infty} a_{k}(m) = 
\left( 3 \, \left( \tfrac{1}{2} \right)_{k} - \left( \tfrac{3}{2} 
\right)_{k} \right) \frac{1}{k! \, 2^{k}}
\end{equation}
\noindent
and 
\begin{eqnarray*}
\sum_{k=0}^{\infty} \lim\limits_{m \to \infty} a_{k}(m)  & = & 
\sum_{k=0}^{\infty} 
\left( 3 \, \left( \tfrac{1}{2} \right)_{k} - \left( \tfrac{3}{2} 
\right)_{k} \right) \frac{1}{k! \, 2^{k}} \\
& = & 
3 \sum_{k=0}^{\infty} \frac{ \left( \tfrac{1}{2} \right)_{k} \, 2^{-k}}{k!} -
\sum_{k=0}^{\infty} \frac{ \left( \tfrac{3}{2} \right)_{k} \, 2^{-k}}{k!}  \\
& = & 3 (1 - 1/2)^{-1/2} - (1-1/2)^{-3/2} \\
& = & \sqrt{2}.
\end{eqnarray*}

The last step is justified using Tannery's theorem. In
the present case $a_{k}(m)$, given in \eqref{def-ak}, satisfies 
\begin{equation}
|a_{k}(m) | \leq \left( 3 \left( \tfrac{1}{2} \right)_{k} + 
\left( \tfrac{3}{2} \right)_{k} \right) \frac{2^{k}}{k!} 
\frac{(-1-m)_{k}}{(-4m)_{k}}. 
\end{equation}
\noindent
The proof of the inequality 
\begin{equation}
\frac{ (-1-m)_{k}}{(-4m)_{k}} \leq \frac{1}{3^{k}},
\label{last-bound}
\end{equation}
\noindent
is similar to the proof of \eqref{bound-111}. This is then used to verify 
that the hypothesis of Tannery's theorem are satisfied. The 
details are omitted. 

\medskip

A third proof is based on the analysis of a function that resembles 
the formula for $T(m)$. 

\begin{proposition}
For $ 0 \leq x < 1$ define 
\begin{equation}
W_{m}(x) = \sum_{r=0}^{m+1} \binom{2r}{r} \binom{m+1}{r} 
\binom{4m}{r}^{-1} x^{r}.
\end{equation}
\noindent
Then 
\begin{equation}
\lim\limits_{m \to \infty} W_{m}(x) = \frac{1}{\sqrt{1-x}} \text{ and }
\lim\limits_{m \to \infty} x \frac{d}{dx} W_{m}(x) = \frac{x}{2(1-x)^{3/2}}.
\end{equation}
\end{proposition}
\begin{proof}
Note that the sum defining $W_{m}(x)$ can be extended to infinity since 
$\binom{m+1}{r}$ has compact support. The proof now follows from 
\begin{equation*}
W_{m}(x) = \sum_{r=0}^{\infty} \binom{2r}{r} 
\left( \frac{x}{4} \right)^{r} 
\prod_{i=1}^{r} \left( \frac{1 - \tfrac{i-2}{m}}{1 - \tfrac{i-1}{m} }
\right) \rightarrow 
\sum_{r=0}^{\infty} \binom{2r}{r} \left( \frac{x}{4} \right)^{r} = 
\frac{1}{\sqrt{1-x}},
\end{equation*}
\noindent
where the passage to the uniform limit is justified by Weierstrass M-test 
or dominated convergence theorem. The second assertion is immediate.
\end{proof}

\begin{corollary}
The sum $T(m)$ satisfies 
\begin{equation}
\lim\limits_{m \to \infty} T(m) = \frac{2 - \sqrt{2}}{\sqrt{2}}.
\label{limit-tm}
\end{equation}
\end{corollary}
\begin{proof}
This follows from the identity 
\begin{equation}
T(m) = \lim\limits_{x \to 1/2} \frac{1}{2} \frac{d}{dx}
W_{m}(x) - W_{m}(x). 
\end{equation}
\end{proof}

\begin{note}
The function $W_{m}(x)$ can be expressed in hypergeometric form as
\begin{equation}
W_{m}(x) = \pFq21{\tfrac{1}{2}, - 1 - m }{-4m}{4x}.
\end{equation}
\end{note}

\section{The monotonicity of $T(m)$}
\label{sec-mono}

This last section describes the convergence of $T(m)$ to its limit 
given in \eqref{limit-tn}.

\begin{theorem}
The function $T(m)$ is monotone increasing for $ m \geq 2$.
\end{theorem}
\begin{proof}
Let
\begin{equation}
  F(r,m) =  \binom{2 r}{r} \binom{m + 1}{r} \frac{r - 1}{2^r
     \binom{4 m}{r}}.
\end{equation}
\noindent
The proof is based on a recurrence involving $F(r,m)$ that is obtained 
by the WZ-technology as developed in \cite{petkovsek-1996a}. Input the 
hypergeometric function $F(k,m)$ into WZ-package with summing range
from $r = 2$ to $r = n + 1$. The recurrence relations that come as the 
ouput is 
\begin{equation}
 a(n)T(n) - b(n)T(n + 1) + c(n)T(n + 2) + d(n) = 0,
\label{recu-WZ}
\end{equation}
  where
  \[ \begin{array}{ll}
       a ( n) = & 7195230 + 87693273 n + 448856568 n^2 + 1263033897 n^3 +
       2147597568 n^4  \\
& + 2279791176 n^5
 + 1502157312 n^6 + 586779648 n^7 + 121208832 n^8 + 9732096 n^9
     \end{array} \]
  \[ \begin{array}{ll}
       b ( n) = & 9661680 + 123557904 n + 651005760 n^2 + 1865031680 n^3 +
       3206772480 n^4 \\
 & + 3428727552 n^5
        + 2272235520 n^6 + 894167040 n^7 + 187269120 n^8 + 15499264 n^9
     \end{array} \]
  \[ \begin{array}{ll}
       c ( n) = & 3265920 + 41472576 n + 217055232 n^2 + 618806528 n^3 +
       1062162432 n^4 \\
  & +  1139030016 n^5
       + 762052608 n^6 + 305528832 n^7 + 66060288 n^8 + 5767168 n^9
     \end{array} \]
  \[ \begin{array}{ll}
       d(n) = & - 799470 - 5607945 n - 14906040 n^2 - 16808745 n^3 - 2987520
       n^4  
      + 9906360 n^5  \\
    &     + 8025600 n^6 + 1858560 n^7.
     \end{array} \]
  
  Note that $b(n) = a(n) + c(n) + d(n)$, then \eqref{recu-WZ}  becomes
\begin{equation}
  a(n) T(n) - (a(n) + c(n) + d(n)) T(n + 1) + c(n) T( n + 2) + d(n) = 0,
\end{equation}
\noindent
which is written as 
\begin{equation}
 a(n) ( T(n) - T(n + 1)) + d(n) ( 1 - T(n + 1)) = c(n) ( T(n+1) - T(n+2)).
\label{main-recu}
\end{equation}

\begin{lemma}
\label{bound-d}
The polynomial $d(m)$ is nonnegative for $m \geq 2$. 
\end{lemma}
\begin{proof}
Simply observe that 
\begin{eqnarray*}
d(x+2)  & = & 814627800 + 2803521195x + 3780146130x^{2} + 2680435095x^{3}  \\
& & 1098008880x^{4} + 262332600x^{5} + 34045440x^{6} + 1858560x^{7}
\end{eqnarray*}
\noindent
is a polynomial with positive coefficients.
\end{proof}

Theorem \ref{bound-tm} shows that $T(m) < 1$ and with this Lemma \ref{bound-d}
implies 
\begin{equation}
 a(n) ( T(n) - T(n + 1)) \leq  c(n) ( T(n+1) - T(n+2)).
\label{main-recu1}
\end{equation}
\noindent
Assume $T$ is not monotone. Define $N$ as the smallest positive 
integer such that 
\begin{equation}
T(N) > T(N+1). 
\end{equation}
\noindent
Then \eqref{main-recu1} implies 
\begin{equation}
 a(N) ( T(N) - T(N + 1)) \leq  c(N) ( T(N+1) - T(N+2))
\label{main-recu2}
\end{equation}
\noindent
and since $a(N) > 0, \, c(N) > 0$, it follows 
that $T(N+1) > T(N+2)$. Iteration of this argument shows that 
the sequence $\{ T(n): \, n \geq N \}$ is monotonically decreasing. 

\smallskip

Let $\delta_{N} = T(N) - T(N+1) > 0$, then \eqref{main-recu2} yields 
\begin{equation}
T(N+1) - T(N+2) \geq \frac{a(N)}{c(N)} \delta_{N}.
\label{main-recu3}
\end{equation}
\noindent
Iterating this procedure gives 
\begin{equation}
T(N + p) - T(N + p + 1) > \delta_N 
\prod_{i=0}^{p-1} \frac{a(N+i)}{c(N+i)}, \text{ for every } p \in \mathbb{N}.
\end{equation}
\noindent 
This inequality is now impossible as $p \to \infty$, since the left-hand side 
converges to $0$ in view of \eqref{limit-tn} and 
\begin{equation}
\lim\limits_{n \to \infty} \frac{a_{n}}{c_{n}} = \frac{27}{16}
\end{equation}
\noindent
showing that the right-hand side blows up. 
\end{proof}

\section{A conjectured inequality for hypergeometric functions}
\label{sec-conj}

The hypergeometric representation for the function $T(m)$ and the 
monotonicity of $T(m)$ give using \eqref{formula-last},
\begin{multline}
\pFq21{\tfrac{3}{2}, -m-2}{-4m-4}{2} -
\pFq21{\tfrac{3}{2}, -m-1}{-4m}{2} >  \\
3 \left[ 
\pFq21{\tfrac{1}{2}, -m-2}{-4m-4}{2}  -
\pFq21{\tfrac{1}{2}, -m-1}{-4m}{2}  
\right]. \nonumber
\end{multline}

\noindent
This is the special case $x = \tfrac{1}{2}$ of the conjecture given below.

\begin{conjecture}
The inequality 
\begin{multline*}
\pFq21{\tfrac{3}{2}, -m-2}{-4m-4}{4x} -
\pFq21{\tfrac{3}{2}, -m-1}{-4m}{4x} >  \\
3 \left[ 
\pFq21{\tfrac{1}{2}, -m-2}{-4m-4}{4x}  -
\pFq21{\tfrac{1}{2}, -m-1}{-4m}{4x}  
\right] 
\end{multline*}
\noindent 
holds for $x \geq \tfrac{1}{2}$. 
\end{conjecture}

\noindent
\textbf{Acknowledgments}. The last author acknowledges the partial 
support of NSF-DMS 1112656. The second author is a post-doctoral fellow 
and the third and fourth authors are graduate students, funded in part by the 
same grant.

\end{document}